\documentclass[dvipdfmx,final,reqno,12pt]{amsart}
\usepackage[utf8]{inputenc}
\usepackage{amsmath,amssymb,amsthm}
\numberwithin{equation}{section}

\usepackage[margin=30mm]{geometry}

\usepackage{mathtools}
\mathtoolsset{showonlyrefs=true}
\usepackage{ascmac}
\usepackage[dvipdfmx]{color}
\usepackage{comment}
\usepackage{enumerate}
\usepackage[dvipdfmx]{graphicx}
\usepackage{here}
\usepackage{showkeys}
\usepackage{mathrsfs}

\usepackage[svgnames,psnames]{xcolor}
\usepackage[colorlinks,citecolor=DarkGreen,linkcolor=FireBrick,linktocpage,unicode]{hyperref}

\theoremstyle{definition}
\newtheorem{thm}{Theorem}[section]
\newtheorem{pro}[thm]{Proposition}
\newtheorem{lem}[thm]{Lemma}
\newtheorem{cor}[thm]{Corollary}

\newtheorem{dfn}[thm]{Definition}
\newtheorem{ex}[thm]{Example}

\newcommand{\INN}[2]{\langle #1, #2 \rangle}

\newcommand{\bs}[1]{\boldsymbol{#1}}

\title{A Simons type condition for instability of $F$-Yang-Mills connections}

\author{Kurando Baba}
\address{Department of Mathematics,
Faculty of Science and Technology,
Tokyo University of Science,
Noda, Chiba, 278-8510,
Japan}
\email{kurando.baba@rs.tus.ac.jp}

\author{Kazuto Shintani}
\address{Department of Mathematics, Graduate School of Science and Technology, Tokyo University of Science, Noda, Chiba, 278-8510, Japan}
\email{6121505@ed.tus.ac.jp}


\date{January 11, 2023}

\begin{document}

\maketitle

\begin{abstract}
$F$-Yang-Mills connections
are critical points
of $F$-Yang Mills functional on the space
of connections of a
principal fiber bundle,
which is a generalization
of Yang-Mills connections,
$p$-Yang-Mills connections
and exponential Yang-Mills
connections and so on.
Here, $F$ is a strictly
increasing $C^{2}$-function.
In this paper,
we extend Simons theorem for an instability of Yang-Mills connections
to $F$-Yang-Mills connections.
We derive
a sufficient condition
that any non-flat, $F$-Yang-Mills connection
over convex hypersurfaces
in a Euclidean space
is instable.
In the sphere case,
this condition is expressed by
an inequality with respect to
its dimension and a degree of the differential of the function $F$.
The proofs of the results
are given by extending
Kobayashi-Ohnita-Takeuchi's
calculation
to $F$-Yang-Mills connections.
\end{abstract}

\tableofcontents

\section{Introduction}

A Yang-Mills connection
is a critical point of the Yang-Mills functional
defined on the
space of connections
of any principal
fiber bundle over a connected,
closed Riemannian manifolds.
There are many developments in the theory of
Yang-Mills connections.
On the other hand,
several Yang-Mills type
functionals
were introduced
and critical points of such functionals have been studied,
for example,
$p$-Yang-Mills functional
(Uhlenbeck \cite{Uhlenbeck},
Chen-Zhou \cite{CZ}),
exponential Yang-Mills functional (Matsuura-Urakawa \cite{MU}, Wei \cite{Wei})
and the generalized Yang-Mills-Born-Infeld energy functional
(Sibner-Sibner-Yang \cite{SSY}, Dong-Wei \cite{DW}, Gherghe \cite{Gherghe}).

An $F$-Yang-Mills functional
provides a unified description
of the above functionals
(Jia-Zhou \cite{JZ}, Dong-Wei \cite{DW}).
Here,
$F$
indicates a strictly
increasing $C^{2}$-function defined on $[0,c)$,
$0<c\leq \infty$.
A critical point of the $F$-Yang-Mill functional
is called an $F$-Yang-Mills connection.
The purpose of this paper is to
study the stability of $F$-Yang-Mills connections.
More precisely,
we would like to give a sufficient
condition that
any non-flat, $F$-Yang-Mills connection
is instable,
which is an extension of 
the following Simons theorem
for the instability
of Yang-Mills connections
to $F$-Yang-Mills connections.

\begin{thm}[{\cite{Simons}}]\label{thm:Simons}
For $n>4$,
any non-flat, Yang-Mills connection
over the standard sphere $S^{n}$
is instable.
\end{thm}

From this theorem
the study of
the weak stability for
the usual
Yang-Mills connections
over $S^{n}$
makes sense only for $n\leq 4$.
It is known that
weakly stable
Yang-Mills connections
over the $4$-sphere
are closely related to
self-dual connections
and anti-self-dual connections.
Indeed,
Bourguignon-Lawson \cite[Theorem B]{BL}
proved
that,
in the case when
the structure group
is a specific unitary group,
any weakly stable
Yang-Mills connection
over $S^{4}$
is either self-dual or anti-self-dual.
On the other hand,
the construction
of these connections
were given by
Atiyah-Drinfeld-Hitchin-Manin
\cite{ADHM}.
We expect that
such studies
can be explored for $F$-Yang-Mills
connections.

This paper
contributes to find
a suitable
extension of Theorem \ref{thm:Simons}
for $F$-Yang-Mills connections.
In fact, we derive a Simons type condition
for the instability
of $F$-Yang-Mills connections
over convex hypersurfaces in
a Euclidean space
(see Theorem \ref{thm:FYM_inst_cch}
for the detail).
The proof of this theorem
is given by extending
Kobayashi-Ohnita-Takeuchi's
calculation \cite{KOT}
of the second variation
of the usual Yang-Mills functional.
From Theorem \ref{thm:FYM_inst_cch},
we have an extension of 
Theorem \ref{thm:Simons}
as follows.

\begin{thm}[{Corollary \ref{cor:FYM_Sninstable}}]\label{thm:FYM_Simonstype}
Let $d_{F'}$ denote the degree
of the differential $F'$ defined in Definition \ref{dfn:F'_degree_dfn}.
Assume that the degree $d_{F'}$
is finite.
Then,
for $n>4d_{F'}+4$,
any non-flat, $F$-Yang-Mills connection
over $S^{n}$
is instable.
\end{thm}

Theorem \ref{thm:FYM_Simonstype}
clarifies the importance
of the finiteness
of the degree $d_{F'}$
in order to derive
the Simons type
condition
for the instability
of $F$-Yang-Mills connections.
For the usual Yang-Mills connections,
this result 
coincides with
Theorem \ref{thm:Simons}
because $d_{F'}=0$ holds.
Furthermore,
it can be verified
that
Theorem \ref{thm:FYM_Simonstype}
also coincides with
the instability theorem
for the $p$-Yang-Mills connections,
which was given by
Chen-Zhou \cite[Corollary 4.2]{CZ}.
We can find an alternative formula
of the instability theorem for
$F$-Yang-Mills connections by
Jia-Zhou \cite[Corollary 16]{JZ}.

On the other hand,
in the case when $F'$ has infinite degree,
it is difficult to find a sufficient condition for
the instability of $F$-Yang-Mills connections
under a general setting.
For example,
the $F$-function
corresponding to
the exponential Yang-Mills functional
$\mathscr{Y\!\!M}_{e}$
or the generalized Yang-Mills-Born-Infeld
energy functional
$\mathscr{Y\!\!M}_{\epsilon=-1}$
with minus sign
has infinite degree.
The stability of
critical points
of these functional
were studied by
Matsuura-Urakawa for $\mathscr{Y\!\!M}_{e}$
and by Gherghe for $\mathscr{Y\!\!M}_{\epsilon=-1}$.
For further developments
we study the instability
for the connections over $S^{n}$
by means of
the index formula
stated in
Theorem \ref{thm:sumI_cal}.
In fact, we derive
a certain sufficient condition for
the instability of the connections
by imposing
the boundedness of 
its curvature
(see, Propositions \ref{pro:gYMBI-_inst}
and \ref{pro:eYM-_inst}
for the detail).

The organization of
this paper is as follows:
In Section \ref{sec:Pre},
we review
the basic notions
in the Yang-Mills theory,
which are related to the present paper.
We note that
Kobayashi-Ohnita-Takeuchi \cite{KOT}
studied the instability of Yang-Mills connections
via analysis of the indices
for harmonic forms.
Here, harmonic forms
are defined
as elements in the zero eigenspace
of Hodge-Laplacian.
We recall Bochner-Weitzenb\"ock formula
for the Hodge-Laplacian,
which is needed for our calculation.
In Section \ref{sec:FYM_dfn},
we review the basics
for $F$-Yang-Mills connections.
In Subsection \ref{sec:FYM_dfn_1stV},
we recall the notion
of $F$-Yang-Mills connections
and derive the $F$-Yang-Mills equation,
i.e., the Euler-Lagrange equation
for the $F$-Yang-Mills functional.
Motivated by the $F$-Yang-Mills equation,
we introduce the notion of $F$-harmonic forms
(Definition \ref{dfn:FHarm_dfn}).
In Subsection \ref{sec:FYM_inst_2stV},
we recall the definition
of the instability
of $F$-Yang-Mills connections
and show the second variational formula
for the $F$-Yang-Mills functional.
Motivated by this formula,
we define the index
of $F$-harmonic forms
(Definition \ref{dfn:FHarm_indx}).
In Section \ref{sec:Simons},
we prove Theorem \ref{thm:FYM_Simonstype}
via analysis of the indices
for $F$-harmonic forms.
In Subsection \ref{sec:FHarm_indx_analysis1},
we extend the result
\cite[(4.37) Proposition]{KOT}
for the index of harmonic forms
to $F$-harmonic forms (Theorem \ref{thm:sumI_cal}).
Following to
Theorem \ref{thm:sumI_cal},
we need essentially to evaluate
a quantity defined in Definition \ref{dfn:h_1_dfn}.
We also find
that the key for proving Theorem 
\ref{thm:FYM_Simonstype}
is to evaluate
the relation between $F'(\|\varphi\|^{2}/2)$
and $F''(\|\varphi\|^{2}/2)$
for an $F$-harmonic $2$-form $\varphi$.
Motivated by this consideration,
in Subsection \ref{sec:FHarm_indx_analysis2},
we introduce
the notion of the degree $d_{F'}$
of the differential $F'$
(Definition \ref{dfn:F'_degree_dfn}).
In the case when $d_{F'}$ is finite,
we derive an inequality
for the index of $F$-harmonic forms
based on Theorem \ref{thm:sumI_cal}
(Theorem \ref{thm:sumIF_ineq2}).
In Subsection \ref{sec:FYM_inst_CCH},
we extend the result
\cite[(5.3) Theorem]{KOT} for
the instability of Yang-Mills connections
to $F$-Yang-Mills connections
by means of Theorem \ref{thm:sumIF_ineq2}
(Theorem \ref{thm:FYM_inst_cch}).
As a corollary
of Theorem \ref{thm:FYM_inst_cch}
we obtain Theorem \ref{thm:FYM_Simonstype}
(Corollary \ref{cor:FYM_Sninstable}).
We prove
Propositions \ref{pro:gYMBI-_inst}
and \ref{pro:eYM-_inst}.
It is known
that there are
strong similarities
between the theory of Yang-Mills connections
and that of harmonic maps.
Finally,
we discuss a counter part of our results
in the theory of harmonic maps.

\section{Preliminaries}\label{sec:Pre}

Let $(M,g)$ be an $n$-dimensional,
connected, closed Riemannian manifold and $D$ denote the Levi-Civita connection on $M$.
Let $G$ be a compact Lie group
and $\mathfrak{g}$ denote its
Lie algebra.
We write the adjoint
representation of $G$ on $\mathfrak{g}$ as
$\mathrm{Ad}:G\to \mathrm{GL}(\mathfrak{g})$.
Let $P$ be a principal
fiber bundle over $M$
with structure group $G$.
A $\mathfrak{g}$-valued
$1$-form $A$ on $P$
is called a connection
if $A$ is of type $\mathrm{Ad}$
and $A(\tilde{X})=X$ holds
for all $X\in\mathfrak{g}$,
where $\tilde{X}$
denotes the fundamental vector
field on $P$ associated with $X$.
We denote by
$\Omega^{k}_{\mathrm{Ad},\mathrm{hor}}(P,\mathfrak{g})$
the vector space of
horizontal $k$-forms
of type $\mathrm{Ad}$
on $P$ with values
in $\mathfrak{g}$.
It is verified that
the curvature $2$-form
of a connection on $P$
gives an element of $\Omega^{2}_{\mathrm{Ad},\mathrm{hor}}(P,\mathfrak{g})$.
The kernel of a connection on $P$
defines an Ehresmann connection,
that is, a right-invariant,
horizontal distribution
on $P$.
Then, it is known that
this distribution
is integrable
if and only if
the curvature $2$-form of $A$
vanishes.
A connection $A$
is said to be \emph{flat},
if its curvature $2$-form
vanishes.
For any two connections
$A, A'$,
the difference
$A-A'$ is in
$\Omega^{1}_{\mathrm{Ad},\mathrm{hor}}(P,\mathfrak{g})$.
Conversely,
$A+\alpha$
gives another connection on $P$
for all
$\alpha\in\Omega^{1}_{\mathrm{Ad},\mathrm{hor}}(P,\mathfrak{g})$.
Hence
the set $\mathscr{C}_{P}$
of connections 
on $P$
becomes an affine space
over the vector space
$\Omega^{1}_{\mathrm{Ad},\mathrm{hor}}(P,\mathfrak{g})$.

We make use of a different
description of 
connections on $P$.
Denote by
$\mathfrak{g}_{P}=P\times_{\mathrm{Ad}}\mathfrak{g}$
the adjoint bundle of $P$,
that is,
the associated vector bundle
of $P$ with the adjoint
representation $\mathrm{Ad}$ of $G$ on $\mathfrak{g}$.
It follows from
\cite[Theorem 5.13.4]{Hamilton}
that
$\Omega^{k}_{\mathrm{Ad},\mathrm{hor}}(P,\mathfrak{g})$
is canonically
isomorphic with
the vector space
of $k$-forms on $M$
with values in $\mathfrak{g}_{P}$,
which we write $\Omega^{k}(\mathfrak{g}_{P})
=\Gamma(\Lambda^{k}T^{*}M\otimes\mathfrak{g}_{P})$.
Any connection on $P$
corresponds to
a connection
on $\mathfrak{g}_{P}$,
i.e., 
a covariant derivative
$\nabla:\Gamma(\mathfrak{g}_{P})\to
\Omega^{1}(\mathfrak{g}_{P})$
on the bundle $\mathfrak{g}_{P}$.
It is shown that
the curvature $R^{\nabla}$
of $\nabla$ on $\mathfrak{g}_{P}$
is in $\Omega^{2}(\mathfrak{g}_{P})$
(cf.~\cite[Corollary 5.13.5]{Hamilton}).
In what follows,
we identify
$\mathscr{C}_{P}$
with the set of
connections on $\mathfrak{g}_{P}$,
which is an affine space
over the vector space $\Omega^{1}(\mathfrak{g}_{P})$.

We give
a fiber metric
on $\mathfrak{g}_{P}$
which is compatible
with connections on $\mathfrak{g}_{P}$.
Such a fiber metric
is induced from
an $\mathrm{Ad}(G)$-invariant
inner product $\INN{\cdot}{\cdot}$ on $\mathfrak{g}$
(cf.~\cite[Proposition 5.9.7]{Hamilton}).
In addition,
$\INN{\cdot}{\cdot}$
also induces a pointwise inner product
on the space $\Omega^{k}(\mathfrak{g}_{P})$, which is denoted by the same symbol $\INN{\cdot}{\cdot}$.
We set
$\|\varphi\|^{2}=\INN{\varphi}{\varphi}$
for $\varphi\in\Omega^{k}(\mathfrak{g}_{P})$.
Here, we write 
$\INN{\varphi}{\psi}$
($\varphi,\psi\in\Omega^{k}(\mathfrak{g}_{P})$) by means of their components.
We take an orthonormal
basis $(e_{1},\dotsc,e_{n})$
of the tangent space $T_{x}M$
at $x\in M$,
and denote by
$(\theta^{1},\dotsc,\theta^{n})$
its dual basis.
If we write
\begin{equation}
\varphi=\dfrac{1}{k!}\sum_{i_{1},\dotsc,i_{k}}
\varphi_{e_{i_{1}},\dotsc,e_{i_{k}}}\theta^{i_{1}}\wedge\dotsb\wedge\theta^{i_{k}}\,,\quad
\psi=\dfrac{1}{k!}\sum_{i_{1},\dotsc,i_{k}}\psi_{e_{i_{1}},\dotsc,e_{i_{k}}}\theta^{i_{1}}\wedge\dotsb\wedge\theta^{i_{k}}\,,
\end{equation}
then we obtain
\begin{equation}
\INN{\varphi}{\psi}=\dfrac{1}{k!}\sum_{i_{1},\dotsc,i_{k}}
\INN{\varphi_{e_{i_{1}},\dotsc,e_{i_{k}}}}{\psi_{e_{i_{1}},\dotsc,e_{i_{k}}}}\,.
\end{equation}
By integrating this pointwise inner product over $M$,
we obtain an inner product on $\Omega^{k}(\mathfrak{g}_{P})$
as follows:
\begin{equation}\label{eqn:()_dfn}
(\varphi,\psi)=\int_{M}\INN{\varphi}{\psi}dv\,,
\end{equation}
where $dv$ denotes the Riemannian volume form on $M$.

For any connection $\nabla$,
the covariant exterior derivative
$d^{\nabla}:\Omega^{k}(\mathfrak{g}_{P})\to\Omega^{k+1}(\mathfrak{g}_{P})$ is given by
\begin{equation}\label{eqn:dnab_nab}
(d^{\nabla}\varphi)_{X_{0},\dotsc,X_{k}}
=\sum_{i=0}^{k}(-1)^{i}(\nabla_{X_{i}}\varphi)_{X_{0},\dotsc,\hat{X}_{i},\dotsc,X_{k}}\,,
\end{equation}
for $\varphi\in\Omega^{k}(\mathfrak{g}_{P})$,
where $X_{0},\dotsc,X_{k}$ are tangent vectors of $M$.
We denote by $\delta^{\nabla}$
the formal adjoint operator of $d^{\nabla}$,
that is, $\delta^{\nabla}:\Omega^{k}(\mathfrak{g}_{P})\to\Omega^{k-1}(\mathfrak{g}_{P})$ is defined by
$(d^{\nabla}\psi,\varphi)
=(\psi,\delta^{\nabla}\varphi)$
for
$\varphi\in\Omega^{k}(\mathfrak{g}_{P})$
and $\psi\in\Omega^{k-1}(\mathfrak{g}_{P})$.
Following to \cite[(2.13)]{BL},
for any $\varphi\in\Omega^{k}(\mathfrak{g}_{P})$,
$\delta^{\nabla}$ has the following expression:
\begin{equation}
(\delta^{\nabla}\varphi)_{X_{1},\dotsc,X_{k-1}}
=-\sum_{j=1}^{n}(\nabla_{e_{j}}\varphi)_{e_{j},X_{1},\dotsc,X_{k-1}}\,.
\end{equation}
For any connection $\nabla$,
the curvature $2$-form
$R^{\nabla}$
satisfies
$d^{\nabla}R^{\nabla}=0$,
which is called the \emph{Bianchi identity} for $\nabla$.
In general, $d^{\nabla}\circ d^{\nabla}$
does not vanish.
It is known that,
if $\nabla$ is flat,
then $d^{\nabla}\circ d^{\nabla}=0$ holds.
This is an alternative interpretation
of flat connections.

A Yang-Mills connection $\nabla$
is defined as a critical point of
the Yang-Mills functional
\begin{equation}
\mathscr{C}_{P}\to\mathbb{R};~
\nabla\mapsto\dfrac{1}{2}\int_{M}\|R^{\nabla}\|^{2}dv\,.
\end{equation}
It is shown that
the Euler-Lagrange equation for this functional
is given by
$\delta^{\nabla}R^{\nabla}=0$,
which is called
the Yang-Mills equation.
Hodge-Laplacian
is defined by $\Delta^{\nabla}=d^{\nabla}\delta^{\nabla}+\delta^{\nabla}d^{\nabla}$,
which gives a differential operator on $\Omega^{k}(\mathfrak{g}_{P})$.
A $\mathfrak{g}_{P}$-valued form $\varphi$
is called a harmonic form
if $\varphi$ satisfies
$\Delta^{\nabla}\varphi=0$.
Then, it is verified that
$\varphi$ is harmonic
if and only if it satisfies
$d^{\nabla}\varphi=0$
and $\delta^{\nabla}\varphi=0$.
This yields
that
the curvature form $R^{\nabla}$
of a Yang-Mills connection $\nabla$
is a harmonic form.
In Section \ref{sec:FYM_dfn},
we will recall the notion
of $F$-Yang-Mills connections,
which is an extension of Yang-Mills connections.
Furthermore,
we will introduce
the notion of $F$-harmonic forms
as an extension of harmonic forms 
(see Definition \ref{dfn:FHarm_dfn}).

We show Bochner-Weitzenb\"ock formula
for $\mathfrak{g}_{P}$-valued forms,
which describes the relation
between
the Hodge-Laplacian and the rough Laplacian.
This formula plays a fundamental role
in analysis of $F$-harmonic forms.
In fact,
we make use of this formula
to prove Proposition \ref{pro:FHarm_BW} in Section \ref{sec:Simons},
which gives a method
to calculate the differential
of the curvature $R^{\nabla}$.
We first recall the notion of the rough Laplacian,
namely, it is defined by
\begin{equation}
\nabla^{*}\nabla\varphi=-\sum_{j=1}^{n}\nabla^{2}_{e_{j},e_{j}}\varphi,
\quad
\varphi\in\Omega^{k}(\mathfrak{g}_{P})\,,
\end{equation}
where $\nabla_{X,Y}^{2}=\nabla_{X}\nabla_{Y}-\nabla_{D_{X}Y}$.
It is verified that
$\nabla^{*}\nabla$ is symmetric and non-negative.
This implies that
a $\mathfrak{g}_{P}$-valued form $\varphi$
satisfies
$\nabla^{*}\nabla\varphi=0$
if and only if
$\varphi$ is parallel ($\nabla\varphi=0$).
We also recall Weitzenb\"ock curvature $\mathfrak{R}^{\nabla}:\Omega^{k}(\mathfrak{g}_{P})\to\Omega^{k}(\mathfrak{g}_{P})$ for $k=1,2$ as follows.

In the case when $k=1$,
the operator 
$\mathfrak{R}^{\nabla}:\Omega^{1}(\mathfrak{g}_{P})\to \Omega^{1}(\mathfrak{g}_{P})$
is given by
\begin{equation}\label{eqn:WCurv_order1}
\mathfrak{R}^{\nabla}(\alpha)
=\sum_{i,j}[R^{\nabla}_{ji},\alpha_{j}]\theta^{i}\,,
\end{equation}
for $\alpha=\sum_{i}\alpha_{i}\theta^{i}\in\Omega^{1}(\mathfrak{g}_{P})$.
If we set
\begin{equation}
[\cdot\wedge\cdot]:\Omega^{1}(\mathfrak{g}_{P})\times\Omega^{1}(\mathfrak{g}_{P})\to\Omega^{2}(\mathfrak{g}_{P});~
[\alpha\wedge\beta]_{X,Y}=[\alpha_{X},\beta_{Y}]-[\alpha_{Y},\beta_{X}]
\end{equation}
then
the following relation holds:
\begin{equation}\label{eqn:Rnabaa_aaR}
\INN{\mathfrak{R}^{\nabla}(\alpha)}{\alpha}=\INN{[\alpha\wedge\alpha]}{R^{\nabla}}\,,\quad
\alpha\in\Omega^{1}(\mathfrak{g}_{P})\,.
\end{equation}
We denote by $R$ the Riemannian curvature on $M$,
and by $\mathrm{Ric}:T_{x}M\to T_{x}M$
($x\in M$) the Ricci curvature operator,
that is,
\begin{equation}
\mathrm{Ric}(X)=\sum_{i=1}^{n}R_{X,e_{i}}e_{i}\,,\quad
X\in T_{x}M\,.
\end{equation}
For
$\alpha\in\Omega^{1}(\mathfrak{g}_{P})$,
we define
$\alpha\circ \mathrm{Ric}\in\Omega^{1}(\mathfrak{g}_{P})$
by
$(\alpha\circ\mathrm{Ric})_{X}
=\alpha_{\mathrm{Ric}(X)}$ for all $X\in T_{x}M$.
Then, 
the following proposition shows
the Bochner-Weitzenb\"ock formula
for $\Omega^{1}(\mathfrak{g}_{P})$.

\begin{pro}[{\cite[(3.2) Theorem]{BL}}]
For $\alpha\in\Omega^{1}(\mathfrak{g}_{E})$,
we have
\begin{equation}
\Delta^{\nabla}\alpha
=\nabla^{*}\nabla\alpha+\alpha\circ \mathrm{Ric}+\mathfrak{R}^{\nabla}(\alpha)\,.
\end{equation}
\end{pro}

Next, let us consider the case when $k=2$.
We recall the Weitzenb\"ock curvature for $\Omega^{2}(\mathfrak{g}_{P})$,
that is,
$\mathfrak{R}^{\nabla}:\Omega^{2}(\mathfrak{g}_{P})\to
\Omega^{2}(\mathfrak{g}_{P})$ is given by
\begin{equation}\label{eqn:WCurv_order2}
\mathfrak{R}^{\nabla}(\varphi)_{X,Y}
=\sum_{j=1}^{n}\left\{
[R^{\nabla}_{e_{j},X},\varphi_{e_{j},Y}]-[R^{\nabla}_{e_{j},Y},\varphi_{e_{j},X}]
\right\}\,,
\end{equation}
for $\varphi\in\Omega^{2}(\mathfrak{g}_{P})$,
where $X, Y$ are tangent vector fields on $M$.
We denote by $\mathfrak{X}(M)$
the space of tangent vector fields on $M$,
and by $\Omega^{2}(M)$
the space of $2$-forms on $M$.
For
$\varphi\in\Omega^{2}(\mathfrak{g}_{P})$
and
$\omega\in\Omega^{2}(M)\otimes\mathrm{End}(\mathfrak{X}(M))$,
we set
\begin{equation}\label{eqn:phiomega_dfn}
(\varphi\circ\omega)_{X,Y}=\dfrac{1}{2}\sum_{j=1}^{n}\varphi_{e_{j},\omega_{X,Y}(e_{j})}\,,
\quad
X,Y\in\mathfrak{X}(M)\,.
\end{equation}
Here, we give a concrete example of $\omega$,
which appears in
the Bochner-Weitzenb\"ock formula
for $\Omega^{2}(\mathfrak{g}_{P})$.

\begin{ex}
For any $X,Y\in\mathfrak{X}(M)$,
we set
\begin{equation}\label{eqn:XwedgeY_dfn}
(X\wedge Y)(Z)=\INN{X}{Z}Y-\INN{Y}{Z}X\,,\quad
Z\in\mathfrak{X}(M)\,.
\end{equation}
If $I$ denotes the identity transformation
on $T_{x}M$,
then
$\mathrm{Ric}\wedge I\in\Omega^{2}(M)\otimes\mathrm{End}(\mathfrak{X}(M))$
is defined by
\begin{equation}
(\mathrm{Ric}\wedge I)_{X,Y}
=\mathrm{Ric}(X)\wedge Y+X\wedge \mathrm{Ric}(Y)\,,
\quad
X,Y\in\mathfrak{X}(M)\,.
\end{equation}
\end{ex}

We are ready to show
the Bochner-Weitzenb\"ock formula
for $\Omega^{2}(\mathfrak{g}_{P})$.

\begin{pro}[{\cite[(3.10) Theorem]{BL}}]\label{pro:BWformula}
For $\varphi\in\Omega^{2}(\mathfrak{g}_{P})$,
we have
\begin{equation}\label{eqn:BWformula}
\Delta^{\nabla}\varphi
=\nabla^{*}\nabla\varphi
+\varphi\circ
(\mathrm{Ric}\wedge I+2R)
+\mathfrak{R}^{\nabla}(\varphi)\,.
\end{equation}
\end{pro}

In order to evaluate the second term
in \eqref{eqn:BWformula},
Kobayashi-Ohnita-Takeuchi
\cite[(4.36)]{KOT} introduced
$R(\varphi,\varphi)$
and $\mathrm{Ric}(\varphi,\varphi)$ for $\varphi
=(1/2)\sum_{i,j}\varphi_{ij}\theta^{i}\wedge\theta^{j}
\in\Omega^{2}(\mathfrak{g}_{P})$,
which are defined as follows:
\begin{align}
R(\varphi,\varphi)
&=\sum_{i,j,k,l}R_{ijkl}\INN{\varphi_{ij}}{\varphi_{kl}}\,,\label{eqn:Rvpvp_dfn}\\
\mathrm{Ric}(\varphi,\varphi)
&=\sum_{i,j,k,l}R_{ik}\delta_{jl}\INN{\varphi_{ij}}{\varphi_{kl}}\,,\label{eqn:Ricvpvp_dfn}
\end{align}
where $R_{ijkl}$ and
$R_{ik}$ are the components of
the Riemannian curvature
$R$ and
the Ricci curvature
$\mathrm{Ric}$ on $M$,
respectively,
that is, 
$R(e_{k},e_{l})e_{j}=\sum_{i}R^{i}_{jkl}e_{i}
=\sum_{i}R_{ijkl}e_{i}$ and
$R_{ik}=\sum_{l}R_{lkli}$.
By the definition,
$R(\varphi,\varphi)$ and
$\mathrm{Ric}(\varphi,\varphi)$
are independent of the choice of
$(e_{1},\dotsc,e_{n})$.
Here, we remark that,
in the original definitions
of $R(\varphi,\varphi)$
and $\mathrm{Ric}(\varphi,\varphi)$,
the inner product $(\cdot,\cdot)$
as in \eqref{eqn:()_dfn}
was used instead of $\INN{\cdot}{\cdot}$.

Then we have the following lemma.

\begin{lem}\label{lem:RicwedgeI+2R_cal}
For any $\varphi\in\Omega^{2}(\mathfrak{g}_{P})$,
we have
\begin{equation}\label{eqn:RicwedgeI+2R_cal}
\INN{\varphi\circ(\mathrm{Ric}\wedge I+2R)}{\varphi}
=\mathrm{Ric}(\varphi,\varphi)-\dfrac{1}{2}R(\varphi,\varphi)\,.
\end{equation}
\end{lem}

\begin{proof}
A direct calculation
shows
\begin{align}
\INN{\varphi\circ\mathrm{Ric}\wedge I}{\varphi}
&=\dfrac{1}{2}\sum_{i,j}\INN{(\varphi\circ\mathrm{Ric}\wedge I)_{e_{i},e_{j}}}{\varphi_{e_{i},e_{j}}}\\
&=\dfrac{1}{2}\sum_{i,j,k}
\INN{\varphi_{e_{k},(\mathrm{Ric}\wedge I)_{e_{i},e_{j}}(e_{k})}}{\varphi_{e_{i},e_{j}}}\,.
\end{align}
Then, by using $\varphi_{e_{i},e_{j}}=-\varphi_{e_{j},e_{i}}$, we get
\begin{equation}
\INN{\varphi\circ\mathrm{Ric}\wedge I}{\varphi}
=\sum_{i,j,k}R_{ik}\INN{\varphi_{e_{k},e_{j}}}{\varphi_{e_{i},e_{j}}}
=\mathrm{Ric}(\varphi,\varphi)\,.
\end{equation}
In a similar manner,
we can derive
\begin{equation}
\INN{\varphi\circ 2R}{\varphi}
=-\dfrac{1}{2}\sum_{i,j,k,l}R_{ijkl}\INN{\varphi_{ij}}
{\varphi_{kl}}
=-\dfrac{1}{2}R(\varphi,\varphi)\,.
\end{equation}
Thus, we have the assertion.
\end{proof}

\section{\texorpdfstring{$F$}{F}-Yang-Mills functionals and \texorpdfstring{$F$}{F}-Yang-Mills connections}\label{sec:FYM_dfn}

\subsection{Definition and the first variational formula}\label{sec:FYM_dfn_1stV}

Let $M$ be a connected, closed Riemannian
manifold and $G$ be a compact connected Lie group.
Let $P=P(M,G)$ be a principal fiber bundle
over $M$ with structure group $G$.
We denote by $\mathfrak{g}_{P}$
the adjoint bundle of $P$.
Let $0<c\leq \infty$
and
$F:[0,c)\to\mathbb{R}$ be a
strictly increasing $C^{2}$-function.
We set $\mathbb{R}_{\geq 0}=\{a\in\mathbb{R}\mid a\geq 0\}$.

\begin{dfn}\label{dfn:FYM_dfn}
The \emph{$F$-Yang-Mills functional}
$\mathscr{Y\!\!M}_{F}:\mathscr{C}_{P}\to\mathbb{R}_{\geq 0}$
is defined by
\begin{equation}
\mathscr{Y\!\!M}_{F}(\nabla)
=\int_{M}F(\dfrac{1}{2}\|R^{\nabla}\|^{2})dv\,.
\end{equation}
A connection $\nabla$ on $\mathfrak{g}_{P}$
is called a \emph{$F$-Yang-Mills connection}
if $\nabla$ is a critical point of $\mathscr{Y\!\!M}_{F}$.
Then, its curvature $2$-form $R^{\nabla}$
is also called the $F$-Yang-Mills field
of $\nabla$.
\end{dfn}

For example,
if we take $F(t)=t$,
then the $F$-Yang-Mills
functional coincides with the usual Yang-Mills functional.
Other examples
are found in Uhlenbeck
(\cite{Uhlenbeck}),
Sibner-Sibner-Yang (\cite{SSY})
and Matsuura-Urakawa (\cite{MU}).

\begin{ex}\label{ex:FYM_ex}
(1) Let $p\geq 2$.
If we put $F_{p}(t)=(1/p)(2t)^{p/2}$,
then the $F_{p}$-Yang-Mills functional
coincides with the $p$-Yang-Mills functional.
A critical point
of the $p$-Yang-Mills functional
is called a $p$-Yang-Mills connection
(cf.~\cite{Uhlenbeck}).

(2) Let $\epsilon=\pm 1$.
If we put $F_{\epsilon}(t)=\epsilon\sqrt{1+2\epsilon t}-\epsilon$,
then the $F_{\epsilon}$-Yang-Mills functional
is called the generalized Yang-Mills-Born-Infeld energy functional with sign $\epsilon$.
We call its critical
point
a critical connection of the functional
(cf.~\cite{SSY}).

(3) If we put $F_{e}(t)=e^{t}$,
then the $F_{e}$-Yang-Mills functional
coincides with the exponential Yang-Mills functional.
A critical point
of the exponential Yang-Mills functional
is called an exponential Yang-Mills connection
(cf.~\cite{MU}).
\end{ex}

$F$-Yang-Mills connections
are obtained by solving
the Euler-Lagrange equation
for $\mathscr{Y\!\!M}_{F}$.
Here, we recall the first variational formula
for the functional.

\begin{pro}[{\cite[Lemma 3.1]{DW}, \cite[(11)]{JZ}}]\label{pro:FYM_1stVar}
Let $\nabla^{t}$
($|t|<\varepsilon$)
be a $C^{\infty}$-curve
in $\mathscr{C}_{P}$ with $\nabla^{0}=\nabla$.
If we put
\begin{equation}\label{eqn:alpha_nablat}
\alpha=\dfrac{d}{dt}\!\biggm|_{t=0}\nabla^{t}\in\Omega^{1}(\mathfrak{g}_{P})\,,
\end{equation}
then we have
\begin{equation}
\dfrac{d}{dt}\!\biggm|_{t=0}\mathscr{Y\!\!M}_{F}(\nabla^{t})=\int_{M}\INN{\delta^{\nabla}(F'(\dfrac{1}{2}\|R^{\nabla}\|^{2})R^{\nabla})}{\alpha}dv\,.
\end{equation}
\end{pro}

\begin{proof}
Let $\nabla\in\mathscr{C}_{P}$
and
$\nabla^{t}=\nabla+A^{t}$
be a $C^{\infty}$-curve
in $\mathscr{C}_{P}$
through $\nabla$,
where
$A^{t}\in\Omega^{1}(\mathfrak{g}_{P})$ with $A^{0}=0$.
Then the curvature of $\nabla^{t}$ is given by
\begin{equation*}
 R^{\nabla^{t}}=R^{\nabla}+d^{\nabla}A^{t}+\frac{1}{2}\lbrack A^{t}\wedge A^{t}\rbrack
\,.  
\end{equation*}
By a straightforward calculation, we have
\begin{align*}
\frac{d}{dt}\mathscr{Y\!\!M}_{F}(\nabla^{t}) &=\int_{M}\frac{d}{dt}F(\frac{1}{2}\|R^{\nabla^{t}}\|^{2})dv\\
&=\int_{M}F'(\frac{1}{2}\|R^{\nabla^{t}}\|^{2})\langle\frac{d}{dt}R^{\nabla^{t}},R^{\nabla^{t}}\rangle dv\\
&=\int_{M}F'(\frac{1}{2}\|R^{\nabla^{t}}\|^{2})\langle d^{\nabla}\frac{d}{dt}A^{t}+\lbrack\frac{d}{dt}A^{t}\wedge A^{t}\rbrack,R^{\nabla^{t}}\rangle dv\,.
\end{align*}
Let $\left.\alpha=\dfrac{d}{dt}\right|_{t=0}\nabla^{t}$. The above equality becomes as follows
\begin{equation*}
\frac{d}{dt}\biggm|_{t=0}\mathscr{Y\!\!M}_{F}(\nabla^{t})=\int_{M}F'(\frac{1}{2}\|R^{\nabla}\|^{2})\langle R^{\nabla},d^{\nabla}\alpha\rangle dv=\int_{M}\INN{\delta^{\nabla}(F'(\dfrac{1}{2}\|R^{\nabla}\|^{2})R^{\nabla})}{\alpha}dv\,.
\end{equation*}
Thus, we have complete the proof.
\end{proof}

From Proposition \ref{pro:FYM_1stVar}
we immediately get the Euler-Lagrange equation
for $\mathscr{Y\!\!M}_{F}$
as follows:

\begin{cor}
$\nabla$ is an $F$-Yang-Mills connection
if and only if
$\nabla$ satisfies
\begin{equation}\label{eqn:FYM_ELeq}
\delta^{\nabla}(F'(\dfrac{1}{2}\|R^{\nabla}\|^{2})R^{\nabla})=0\,.
\end{equation}
\end{cor}

We call \eqref{eqn:FYM_ELeq}
the \emph{$F$-Yang-Mills equation}.
Motivated by the $F$-Yang-Mills equation,
we introduce the notion of $F$-harmonic forms as follows.

\begin{dfn}\label{dfn:FHarm_dfn}
A $\mathfrak{g}_{P}$-valued
form $\varphi$
is said to be \emph{$F$-harmonic},
if $\varphi$ satisfies the following two equations:
\begin{equation}\label{eqn:FHarm_dfn}
d^{\nabla}\varphi=0\,,\quad
\delta^{\nabla}(F'(\dfrac{1}{2}\|\varphi\|^{2})\varphi)=0\,.
\end{equation}
\end{dfn}

We note that
the curvature $2$-form $R^{\nabla}$
of an $F$-Yang-Mills connection $\nabla$
is $F$-harmonic.
Indeed, 
$R^{\nabla}$
satisfies
\eqref{eqn:FHarm_dfn}
because of
the Bianchi identity 
and the $F$-Yang Mills equation
for $\nabla$.

\subsection{Instability and the second variational formula}\label{sec:FYM_inst_2stV}

Let us consider
the instability for an $F$-Yang-Mills connection.
We recall here the definition of this property.

\begin{dfn}
An $F$-Yang-Mills connection $\nabla$
is said to be weakly stable
if
the following inequality holds
for any $\alpha\in\Omega^{1}(\mathfrak{g}_{P})$:
\begin{equation}
\dfrac{d^{2}}{dt^{2}}\!\biggm|_{t=0}
\mathscr{Y\!\!M}_{F}(\nabla^{t})
\geq 0\quad
\text{where\quad}
\alpha=\dfrac{d}{dt}\!\biggm|_{t=0}\nabla^{t}\,.
\end{equation}
An $F$-Yang-Mills connection $\nabla$
is said to be \emph{instable}
if $\nabla$ is not weakly stable.
\end{dfn}

The following proposition
gives the second variational formula
for the $F$-Yang-Mills functional.

\begin{pro}\label{pro:FYM_2ndVar}
Let $\nabla$ be an $F$-Yang-Mills connection
and $\nabla^{t}$
($|t|<\varepsilon$) be a $C^{\infty}$-curve
in $\mathscr{C}_{P}$ with $\nabla^{0}=\nabla$.
Then the second variation of the $F$-Yang-Mills
functional is given by
the following:
\begin{multline}\label{eqn:FYM_2ndVar}
\dfrac{d^{2}}{dt^{2}}\!\biggm|_{t=0}\mathscr{Y\!\!M}_{F}(\nabla^{t})
=\int_{M}
F''(\dfrac{1}{2}\|R^{\nabla}\|^{2})\INN{d^{\nabla}\alpha}{R^{\nabla}}^{2}dv\\
+\int_{M}F'(\dfrac{1}{2}\|R^{\nabla}\|^{2})\left\{\INN{\mathfrak{R}^{\nabla}(\alpha)}{\alpha}
+\|d^{\nabla}\alpha\|^{2}
\right\}dv\,,
\end{multline}
where 
$\alpha=\dfrac{d}{dt}\!\biggm|_{t=0}\nabla^{t}$.
\end{pro}

\begin{proof}
A direct calculation yields
\begin{equation*}
\frac{d}{dt}R^{\nabla^{t}}=d^{\nabla}\frac{dA^{t}}{dt}+\frac{1}{2}\frac{d}{dt}\lbrack A^{t}\wedge A^{t}\rbrack\,,
\end{equation*}
and
\begin{equation*}
\frac{d^{2}}{dt^{2}}R^{\nabla^{t}}=d^{\nabla}(\frac{d^{2}}{dt^{2}}A^{t})+\lbrack\frac{d^{2}}{dt^{2}}A^{t}\wedge A^{t}\rbrack+\lbrack\frac{dA^{t}}{dt}\wedge\frac{dA^{t}}{dt}\rbrack\,.
\end{equation*}
Hence we have 
\begin{equation*}
\left.\frac{d}{dt}\right|_{t=0}R^{\nabla^{t}}=d^{\nabla}\alpha,\quad \left.\frac{d^{2}}{dt^{2}}\right|_{t=0}=d^{\nabla}\beta+\lbrack\alpha\wedge\alpha\rbrack. 
\end{equation*}
where
$\alpha=\dfrac{d}{dt}\!\biggm|_{t=0}A^{t}$ and
$\beta=\left.\dfrac{d^{2}}{dt^{2}}\right|_{t=0}A^{t}$. We have

\begin{multline}\label{eqn:YMF_cal}
\left.\frac{d^{2}}{dt^{2}}\right|_{t=0}\mathscr{Y\!\!M}_{F}(\nabla^{t})\\
=\int_{M}F''(\frac{1}{2}\|R^{\nabla}\|^{2})\langle d^{\nabla}\alpha,R^{\nabla}\rangle^{2}dv
+\int_{M}F'(\frac{1}{2}\|R^{\nabla}\|^{2})
\left\{\langle 
\lbrack\alpha\wedge\alpha\rbrack,R^{\nabla}\rangle
+\|d^{\nabla}\alpha\|^{2}\right\}dv\\
+\int_{M}F'(\frac{1}{2}\|R^{\nabla}\|^{2})
\langle 
d^{\nabla}\beta,R^{\nabla}\rangle
dv.
\end{multline}
Then 
it can be verified
that
the third term
of \eqref{eqn:YMF_cal}
vanishes.
Indeed,
since $\nabla$
is an $F$-Yang-Mills connection,
we find
\begin{equation*}
    \int_{M}F'(\frac{1}{2}\|R^{\nabla}\|^{2})\langle d^{\nabla}\beta,R^{\nabla}\rangle dv=\int_{M}\langle\beta,\delta^{\nabla}(F'(\frac{1}{2}\|R^{\nabla}\|^{2})R^{\nabla})\rangle dv=0.
\end{equation*}
Therefore, we obtain
\begin{align*}
&\left.\frac{d^{2}}{dt^{2}}\right|_{t=0}\mathscr{Y\!\!M}_{F}(\nabla^{t})\\
&=\int_{M}F''(\frac{1}{2}\|R^{\nabla}\|^{2})\langle d^{\nabla}\alpha,R^{\nabla}\rangle^{2}dv
+\int_{M}F'(\frac{1}{2}\|R^{\nabla}\|^{2})\left\{\langle\lbrack\alpha\wedge\alpha\rbrack,R^{\nabla}\rangle +\|d^{\nabla}\alpha\|^{2}\right\}dv\\
&=\int_{M}F''(\frac{1}{2}\|R^{\nabla}\|^{2})\langle d^{\nabla}\alpha,R^{\nabla}\rangle^{2}dv
+\int_{M}F'(\frac{1}{2}\|R^{\nabla}\|^{2})\left\{\langle\mathfrak{R}^{\nabla}(\alpha),\alpha\rangle
+\|d^{\nabla}\alpha\|^{2}\right\}dv.
\end{align*}
Here,
in the last equality
we have used \eqref{eqn:Rnabaa_aaR}.
Thus, we have complete the proof.
\end{proof}

An alternative expression of the second variational
formula is found in \cite[(20)]{JZ}.
The difference between them
is the integrand of the second term of \eqref{eqn:FYM_2ndVar}.
In the case when
$\mathscr{Y\!\!M}_{F}$
is the usual Yang-Mills functional
($F(t)=t$),
$F''(t)=0$ holds,
so that the first term
of \eqref{eqn:FYM_2ndVar}
vanishes.

Motivated by Proposition \ref{pro:FYM_2ndVar},
we define the index for
any $F$-harmonic $2$-form as follows:

\begin{dfn}\label{dfn:FHarm_indx}
The \emph{index}
of an $F$-harmonic form $\varphi\in\Omega^{2}(\mathfrak{g}_{P})$
is defined by
\begin{equation}
I_{\varphi}(\alpha)
=\int_{M}F''(\dfrac{1}{2}\|\varphi\|^{2})\INN{d^{\nabla}\alpha}{\varphi}^{2}dv
+\int_{M}F'(\dfrac{1}{2}\|\varphi\|^{2})\left\{
\INN{\mathfrak{R}^{\nabla}(\alpha)}{\alpha}+
\|d^{\nabla}\alpha\|^{2}
\right\}dv\,,
\end{equation}
for any $\alpha\in\Omega^{1}(\mathfrak{g}_{P})$.
\end{dfn}

It follows from Proposition \ref{pro:FYM_2ndVar} that,
for any $F$-Yang-Mills connection $\nabla$,
if $\nabla$ is weakly stable,
then $I_{R^{\nabla}}(\alpha)\geq 0$
holds for all $\alpha\in\Omega^{1}(\mathfrak{g}_{P})$.
In the next section,
we will derive a sufficient condition
for the instability of
$F$-Yang-Mills connections
via analysis of the indices
for $F$-harmonic forms.

\section{A Simons type condition for instability of \texorpdfstring{$F$}{F}-Yang-Mills connections}\label{sec:Simons}

\subsection{Analysis of the indices for \texorpdfstring{$F$}{F}-harmonic forms (1)}\label{sec:FHarm_indx_analysis1}

Let $M$ be an $n$-dimensional,
connected, closed Riemannian manifold
and $D$ denote the Levi-Civita connection on $M$.
Let $P=P(M,G)$ be a principal fiber bundle
over $M$ with structure group $G$.
Suppose that
the base space
$M$ is isometrically immersed in an $N$-dimensional
Euclidean space $(\mathbb{R}^{N},\INN{\cdot}{\cdot})$
with $n<N$.
Denote by $h$ its second fundamental form.
We shall make use of the following convention
on the ranges of indices:
\begin{equation}
1\leq A, B, C \leq N,\quad
1\leq i,j,k,l, m \leq n,\quad
n+1 \leq \mu \leq N\,.
\end{equation}
Let $(e_{1},\dotsc,e_{n})$ be an orthonormal basis
of $T_{x}M$ ($x\in M$)
and $(e_{n+1},\dotsc,e_{N})$ be an orthonormal basis
of the normal space $T^{\perp}_{x}M$ of $M$ in $\mathbb{R}^{N}$.
Let $(E_{1},\dotsc,E_{N})$
be the canonical basis of $\mathbb{R}^{N}$.
We denote by $V_{A}$
the tangent component of $E_{A}$
with respect to the orthogonal decomposition
$\mathbb{R}^{N}=T_{x}M\oplus T^{\perp}_{x}M$.
If
we set $v_{A}^{B}=\INN{E_{A}}{e_{B}}$,
then the matrix
$(v_{A}^{B})_{1\leq A,B\leq N}$
becomes orthogonal.
The tangent vector field $V_{A}$
is given by $V_{A}=\sum_{i}v_{A}^{i}e_{i}$.
Let
$h^{\mu}_{ij}$
denote the component of
$h(e_{i},e_{j})=\sum_{\mu}h^{\mu}_{ij}e_{\mu}$.
Then
we get
the following lemma.

\begin{lem}
With the above settings,
we obtain:
\begin{equation}\label{eqn:DVA_form}
D_{e_{i}}V_{A}
=\sum_{j}\sum_{\mu}v_{A}^{\mu}h_{ij}^{\mu}e_{j}\,.
\end{equation}
\end{lem}

\begin{proof}
We write
$D_{e_{i}}V_{A}$
as
$D_{e_{i}}V_{A}
=\sum_{j}\INN{D_{e_{i}}V_{A}}{e_{j}}e_{j}$.
In order to prove
this lemma,
it is sufficient to 
verify $\INN{D_{e_{i}}V_{A}}{e_{j}}=\INN{E_{A}}{h(e_{i},e_{j})}$.
Since $M$ is isometrically immersed
in $(\mathbb{R}^{N},\INN{\cdot}{\cdot})$,
the Levi-Civita connection $D$ on $M$
is compatible with
$\INN{\cdot}{\cdot}$.
Hence we have
\begin{equation}\label{eqn:DeVe1}
e_{i}\INN{V_{A}}{e_{j}}
=\INN{D_{e_{i}}V_{A}}{e_{j}}+\INN{V_{A}}{D_{e_{i}}e_{j}}\,.
\end{equation}
On the other hand,
since $E_{A}$
is parallel with respect to 
the canonical connection
$D^{0}$ on $\mathbb{R}^{N}$,
we have
\begin{equation}\label{eqn:DeVe2}
e_{i}\INN{V_{A}}{e_{j}}
=e_{i}\INN{E_{A}}{e_{j}}
=\INN{E_{A}}{D^{0}_{e_{i}}e_{j}}
=\INN{V_{A}}{D_{e_{i}}e_{j}}+
\INN{E_{A}}{h(e_{i},e_{j})}\,.
\end{equation}
Here, in the last equality,
we have used the Gauss formula for the submanifold $M$ in $\mathbb{R}^{N}$.
By comparing
\eqref{eqn:DeVe1}
to
\eqref{eqn:DeVe2},
we get
$\INN{D_{e_{i}}V_{A}}{e_{j}}=\INN{E_{A}}{h(e_{i},e_{j})}$.
Thus, we have complete the proof.
\end{proof}

We evaluate
the indices for
$F$-harmonic $2$-forms.
More precisely,
we calculate the summation $\sum_{A}I_{\varphi}(\iota_{V_{A}}\varphi)$
for an $F$-harmonic $2$-form $\varphi$,
where $\iota$ denotes the interior product.
By Definition \ref{dfn:FHarm_indx}
we have
\begin{multline}\label{eqn:sumIVA_1+}
\sum_{A}I_{\varphi}(\iota_{V_{A}}\varphi)
=\int_{M}
F''(\dfrac{1}{2}\|\varphi\|^{2})\INN{\sum_{A}d^{\nabla}(\iota_{V_{A}}\varphi)}{\varphi}^{2}dv\\
+\int_{M}F'(\dfrac{1}{2}\|\varphi\|^{2})\left\{\INN{\sum_{A}\mathfrak{R}^{\nabla}(\iota_{V_{A}}\varphi)}{\iota_{V_{A}}\varphi}
+\sum_{A}\|d^{\nabla}(\iota_{V_{A}}\varphi)\|^{2}
\right\}dv\,.
\end{multline}

Following to \cite[(4.37)]{KOT},
we define $H(\varphi,\varphi)$ for any $\varphi\in\Omega^{2}(\mathfrak{g}_{P})$
as follows :
\begin{equation}\label{eqn:H_dfn}
H(\varphi,\varphi)
=\sum_{i,j,k,l}\sum_{\mu}H^{\mu}h_{ik}^{\mu}\delta_{jl}\INN{\varphi_{ij}}{\varphi_{kl}}\,,
\end{equation}
where 
$H^{\mu}=\sum_{m}h^{\mu}_{mm}$ denotes
the mean curvature of $M$ in $\mathbb{R}^{N}$.

On the other hand, in the present case
we introduce the following quantity.

\begin{dfn}\label{dfn:h_1_dfn}
We set 
\begin{equation}\label{eqn:h_1_dfn}
\bs{h}_{1}(\varphi,\varphi)
=\sum_{\mu}h_{1}^{\mu}(\varphi,\varphi)e_{\mu}\,,\quad
h_{1}^{\mu}(\varphi,\varphi)=
\sum_{i,j,k,l}h_{ik}^{\mu}\delta_{jl}\INN{\varphi_{ij}}{\varphi_{kl}}\,.
\end{equation}
\end{dfn}

Here, we note that
$H(\varphi,\varphi)$ and
$\bs{h}_{1}(\varphi,\varphi)$
are independent of the choice
of $(e_{1},\dotsc,e_{n})$ and $(e_{n+1},\dotsc,e_{N})$.
In addition, for each $\mu$,
the component
$h^{\mu}_{1}(\varphi,\varphi)$
of $\bs{h}_{1}(\varphi,\varphi)$
is also independent of the choice
of $(e_{1},\dotsc,e_{n})$.
As shown later in Theorem \ref{thm:sumI_cal},
$\bs{h}_{1}(\varphi,\varphi)$
is needed to evaluate the first term in
\eqref{eqn:sumIVA_1+}.

The purpose of this subsection
is to prove the following theorem.

\begin{thm}\label{thm:sumI_cal}
With the above settings,
we obtain:
\begin{multline}\label{eqn:sumI_cal}
\sum_{A}I_{\varphi}(\iota_{V_{A}}\varphi)
=\int_{M}F''(\dfrac{1}{2}\|\varphi\|^{2})\INN{\bs{h}_{1}(\varphi,\varphi)}{\bs{h}_{1}(\varphi,\varphi)}dv\\
+\int_{M}F'(\dfrac{1}{2}\|\varphi\|^{2})\left\{
H(\varphi,\varphi)-2\mathrm{Ric}(\varphi,\varphi)+R(\varphi,\varphi)
\right\}dv\,.
\end{multline}
\end{thm}

This theorem
is an extension of
\cite[(4.37) Proposition]{KOT} to $F$-harmonic
forms.

In order to prove
Theorem \ref{thm:sumI_cal},
we first prepare some results.

\begin{lem}\label{lem_sumR_R}
For any $\varphi\in\Omega^{2}(\mathfrak{g}_{P})$,
we have
\begin{equation}\label{eqn:Rvapvap_Rpp}
\sum_{A}\INN{\mathfrak{R}^{\nabla}(\iota_{V_{A}}\varphi)}{\iota_{V_{A}}\varphi}
=\INN{\mathfrak{R}^{\nabla}(\varphi)}{\varphi}\,.
\end{equation}
\end{lem}

\begin{proof}
We express $\varphi$ and $R^{\nabla}$ as follows:
\begin{equation}
\varphi=\dfrac{1}{2}\sum_{i,j}\varphi_{e_{i},e_{j}}\theta^{i}\wedge\theta^{j},\quad
R^{\nabla}
=\dfrac{1}{2}\sum_{i,j}
R^{\nabla}_{e_{i},e_{j}}\theta^{i}\wedge\theta^{j}\,.
\end{equation}
Then, by \eqref{eqn:Rnabaa_aaR},
we have
\begin{align}
\text{the l.h.s.~of \eqref{eqn:Rvapvap_Rpp}}
&= \sum_{A}\INN{[\iota_{V_{A}}\varphi\wedge\iota_{V_{A}}\varphi]}{R^{\nabla}}\\
&=\dfrac{1}{2}\sum_{A}\sum_{i,j}
2\INN{[(\iota_{V_{A}}\varphi)_{e_{i}},(\iota_{V_{A}}\varphi)_{e_{j}}]}{R^{\nabla}_{e_{i},e_{j}}}\\
&=\sum_{i,j,k}\INN{[\varphi_{e_{k},e_{i}},\varphi_{e_{k},e_{j}}]}{R^{\nabla}_{e_{i},e_{j}}}\,.\label{eqn:Rvapvap_Rpp_lft}
\end{align}
Here, in the last equality,
we have used
$(\iota_{V_{A}}\varphi)_{e_{i}}
=\sum_{k}v_{A}^{k}\varphi_{e_{k},e_{i}}$.
On the other hand,
by the $\mathrm{Ad}(G)$-invariance of $\INN{\cdot}{\cdot}$,
we get the following two relations:
\begin{align}
\INN{[R^{\nabla}_{e_{i},e_{k}},\varphi_{e_{i},e_{j}}]}{\varphi_{e_{j},e_{k}}}
&=-\INN{\varphi_{e_{i},e_{j}}}{[R^{\nabla}_{e_{i},e_{k}},\varphi_{e_{j},e_{k}}]}\,,\\
\INN{[R^{\nabla}_{e_{i},e_{j}},\varphi_{e_{i},e_{k}}]}{\varphi_{e_{j},e_{k}}}
&=
\INN{[\varphi_{e_{k},e_{i}},\varphi_{e_{k},e_{j}}]}{R^{\nabla}_{e_{i},e_{j}}}\,.
\end{align}
By using these relations,
we obtain
\begin{align}
\text{
the r.h.s.~of
\eqref{eqn:Rvapvap_Rpp}}
&=\dfrac{1}{2}\sum_{i,j,k}
\left\{
\INN{[R^{\nabla}_{e_{i},e_{j}},\varphi_{e_{i},e_{k}}]}{\varphi_{e_{j},e_{k}}}
-\INN{[R^{\nabla}_{e_{i},e_{k}},\varphi_{e_{i},e_{j}}]}{\varphi_{e_{j},e_{k}}}
\right\}\\
&=\dfrac{1}{2}\sum_{i,j,k}
\left\{
\INN{[R^{\nabla}_{e_{i},e_{j}},\varphi_{e_{i},e_{k}}]}{\varphi_{e_{j},e_{k}}}
+\INN{\varphi_{e_{i},e_{j}}}{[R^{\nabla}_{e_{i},e_{k}},\varphi_{e_{j},e_{k}}]}
\right\}\\
&=\sum_{i,j,k}\INN{[R^{\nabla}_{e_{i},e_{j}},\varphi_{e_{i},e_{k}}]}{\varphi_{e_{j},e_{k}}}\\
&=\sum_{i,j,k}\INN{[\varphi_{e_{k},e_{i}},\varphi_{e_{k},e_{j}}]}{R^{\nabla}_{e_{i},e_{j}}}\,.\label{eqn:Rvapvap_Rpp_rgt}
\end{align}
Comparing \eqref{eqn:Rvapvap_Rpp_lft}
to \eqref{eqn:Rvapvap_Rpp_rgt}
we have the assertion.
\end{proof}

We define
$h_{2}(\varphi,\varphi)$
and $h_{2}'(\varphi,\varphi)$
for $\varphi\in\Omega^{2}(\mathfrak{g}_{E})$
as follows:
\begin{equation}
h_{2}(\varphi,\varphi)
=\sum_{i,j,k,l}\sum_{\mu}h_{ik}^{\mu}h_{lj}^{\mu}\INN{\varphi_{ij}}{\varphi_{kl}},
\quad
h'_{2}(\varphi,\varphi)
=\sum_{i,j,k,l,m}\sum_{\mu}
h_{mk}^{\mu}h_{mi}^{\mu}\delta_{jl}\INN{\varphi_{ij}}{\varphi_{kl}}\,.
\end{equation}
By the definition,
$h(\varphi,\varphi)$ and
$h_{2}'(\varphi,\varphi)$
are independent of the choice
of $(e_{1},\dotsc,e_{n})$ and $(e_{n+1},\dotsc,e_{N})$.
Then, we have
the following lemma.

\begin{lem}\label{lem:sumdvp_sumdvpp}
Let $\varphi$ be in $\Omega^{2}(\mathfrak{g}_{P})$
satisfying $d^{\nabla}\varphi=0$.
Then we have:
\begin{enumerate}[(1)]
\item $\displaystyle\sum_{A}\|d^{\nabla}(\iota_{V_{A}}\varphi)\|^{2}
=\|\nabla\varphi\|^{2}+h_{2}(\varphi,\varphi)
+h_{2}'(\varphi,\varphi)$.
\item[]
\item $\displaystyle\sum_{A}\INN{d^{\nabla}(\iota_{V_{A}}\varphi)}{\varphi}^{2}
=\|\varphi\|^{2}\|\nabla\|\varphi\|\|^{2}
+\INN{\bs{h}_{1}(\varphi,\varphi)}{\bs{h}_{1}(\varphi,\varphi)}$.
\end{enumerate}
\end{lem}

\begin{proof}
Let $\varphi\in\Omega^{2}(\mathfrak{g}_{P})$
with $d^{\nabla}\varphi=0$.
We express $\varphi$ and $\nabla\varphi$
as follows:
\begin{equation}
\varphi=\dfrac{1}{2}\sum_{i,j}\varphi_{ij}\theta^{i}\wedge\theta^{j}\,,
\quad
\nabla\varphi
=\dfrac{1}{2}\sum_{i,j,k}\nabla_{k}\varphi_{ij}\theta^{k}\otimes(\theta^{i}\wedge\theta^{j})\,.
\end{equation}
(1) If we write
\begin{equation}
d^{\nabla}(\iota_{V_{A}}\varphi)=\dfrac{1}{2}\sum_{i,j}(d^{\nabla}(\iota_{V_{A}}\varphi))_{e_{i},e_{j}}\theta^{i}\wedge\theta^{j}\,,
\end{equation}
then, by means of \eqref{eqn:dnab_nab},
the component $(d^{\nabla}(\iota_{V_{A}}\varphi))_{e_{i},e_{j}}$
has the following expression:
\begin{multline}\label{eqn:dnabala_1_1}
(d^{\nabla}(\iota_{V_{A}}\varphi))_{e_{i},e_{j}}
=
\sum_{k}\sum_{\mu} v_{A}^{\mu}h_{ik}^{\mu}\varphi_{kj}
+\sum_{k} v_{A}^{k}(\nabla_{e_{i}}\varphi)_{e_{k},e_{j}}\\
-\left\{\sum_{k}\sum_{\mu} v_{A}^{\mu}h_{jk}^{\mu}\varphi_{ki}
+\sum_{k} v_{A}^{k}(\nabla_{e_{j}}\varphi)_{e_{k},e_{i}}\right\}\,.
\end{multline}
Then, we have
\begin{align}
\sum_{A}\|d^{\nabla}(\iota_{V_{A}}\varphi)\|^{2}
&=\dfrac{1}{2}\sum_{A}\sum_{i,j}\INN{
d^{\nabla}(\iota_{V_{A}}\varphi)_{e_{i},e_{j}}}{
d^{\nabla}(\iota_{V_{A}}\varphi)_{e_{i},e_{j}}}\\
&=\sum_{i,j,k}\left\{
\INN{\nabla_{i}\varphi_{kj}}{\nabla_{i}\varphi_{kj}}-\INN{\nabla_{i}\varphi_{kj}}{\nabla_{j}\varphi_{ki}}\right\}\notag\\
&\phantom{hogehoge}
+\sum_{i,j,k,l}
\sum_{\mu}\left\{
h_{ik}^{\mu}h_{il}^{\mu}
\INN{\varphi_{kj}}{\varphi_{lj}}
-h_{ik}^{\mu}h_{jl}^{\mu}\INN{\varphi_{kj}}{\varphi_{li}}
\right\}\,.\label{eqn:dnabala_1_2}
\end{align}
It can be shown that
the second term of
\eqref{eqn:dnabala_1_2}
coincides with
$h_{2}(\varphi,\varphi)+h_{2}'(\varphi,\varphi)$.
On the other hand,
we make use of 
the condition $d^{\nabla}\varphi=0$
in order to verify
that
the first term of
\eqref{eqn:dnabala_1_2}
is equal to $\|\nabla\varphi\|^{2}$.
Here, 
$d^{\nabla}\varphi=0$
yields
\begin{equation}\label{eqn:nabla_cirlesum}
\nabla_{i}\varphi_{kj}
+\nabla_{j}\varphi_{ik}
+\nabla_{k}\varphi_{ji}=0\,.
\end{equation}
By using this we obtain
\begin{align}
\sum_{i,j,k}\INN{\nabla_{i}\varphi_{kj}}{\nabla_{j}\varphi_{ki}}
&=\sum_{i,j,k}\INN{\nabla_{i}\varphi_{kj}}{\nabla_{i}\varphi_{kj}}
+\sum_{i,j,k}\INN{\nabla_{i}\varphi_{kj}}{\nabla_{k}\varphi_{ji}}\\
&=\sum_{i,j,k}\INN{\nabla_{i}\varphi_{kj}}{\nabla_{i}\varphi_{kj}}
+\sum_{i,j,k}\INN{\nabla_{j}\varphi_{ik}}{\nabla_{i}\varphi_{kj}}\,,
\end{align}
that is,
\begin{equation}
\sum_{i,j,k}\INN{\nabla_{i}\varphi_{kj}}{\nabla_{j}\varphi_{ki}}
=\dfrac{1}{2}\sum_{i,j,k}
\INN{\nabla_{i}\varphi_{kj}}{\nabla_{i}\varphi_{kj}}\,.
\end{equation}
From the above arguments
we obtain (1).

(2) By means of
\eqref{eqn:dnabala_1_1},
we have
\begin{align}
&\sum_{A}\INN{d^{\nabla}(\iota_{V_{A}}\varphi)}{\varphi}\\
&=\sum_{\substack{i,j,k\\ i',j'}}
\INN{\varphi_{ij}}{\nabla_{i}\varphi_{kj}}
\INN{\varphi_{i'j'}}{\nabla_{i'}\varphi_{kj'}}+\sum_{\substack{i,j,k\\ i',j',k'}}
\sum_{\mu}h_{ik}^{\mu}h_{i'k'}^{\mu}\INN{\varphi_{ij}}{\varphi_{kj}}
\INN{\varphi_{i'j'}}{\varphi_{k'j'}}
\\
&=
\sum_{k}
\left(\sum_{i,j}
\INN{\varphi_{ij}}{\nabla_{i}\varphi_{kj}}\right)^{2}+\sum_{\substack{i,j,k,l\\ i',j',k',l'}}
\sum_{\mu}h_{ik}^{\mu}h_{i'k'}^{\mu}\delta_{jl}\delta_{j'l'}\INN{\varphi_{ij}}{\varphi_{kl}}
\INN{\varphi_{i'j'}}{\varphi_{k'l'}}\,.\label{eqn:dnabvapp_1}
\end{align}
Then, we can verify
that
the second term of the right hand side
of
\eqref{eqn:dnabvapp_1}
coincides with
$\INN{\bs{h}_{1}(\varphi,\varphi)}{\bs{h}_{1}(\varphi,\varphi)}$.
By using
\eqref{eqn:nabla_cirlesum},
we get
\begin{equation}
\sum_{i,j}\INN{\varphi_{ij}}{\nabla_{i}\varphi_{kj}}
=\dfrac{1}{2}\sum_{i,j}\INN{\varphi_{ij}}{\nabla_{k}\varphi_{ij}}
=\INN{\varphi}{\nabla_{e_{k}}\varphi}
=\|\varphi\|(\nabla_{e_{k}}\|\varphi\|)\,.
\end{equation}
Substituting this into
the first term of
the right hand side of \eqref{eqn:dnabvapp_1},
we have
\begin{equation}
\sum_{k}
\left(\sum_{i,j}
\INN{\varphi_{ij}}{\nabla_{i}\varphi_{kj}}\right)^{2}
=\|\varphi\|^{2}\sum_{k}\left(
\nabla_{e_{k}}\|\varphi\|
\right)^{2}
=\|\varphi\|^{2}\|\nabla\|\varphi\|\|^{2}\,.
\end{equation}
From the above arguments
we have complete the proof of this lemma.
\end{proof}

Here, we rewrite
$h_{2}(\varphi,\varphi)$
and
$h_{2}'(\varphi,\varphi)$
in terms of
$\mathrm{Ric}(\varphi,\varphi)$,
$R(\varphi,\varphi)$ and
$H(\varphi,\varphi)$.

\begin{lem}\label{lem:h2h2'RHRic}
For any
$\varphi\in\Omega^{2}(\varphi,\varphi)$,
we get:
\begin{equation}\label{eqn:h2h2'RHRic}
h_{2}(\varphi,\varphi)
=\dfrac{1}{2}R(\varphi,\varphi)\,,\quad
h_{2}'(\varphi,\varphi)
=H(\varphi,\varphi)-\mathrm{Ric}(\varphi,\varphi)\,.
\end{equation}
\end{lem}

\begin{proof}
It follows from
the Gauss equation
for $M$ in $\mathbb{R}^{N}$
(\cite[Proposition 4.1, Chapter VII]{KN2})
that the following relation holds:
\begin{equation}\label{eqn:curv_Gauss}
R_{ijkl}=\sum_{\mu}(h_{ik}^{\mu}h_{jl}^{\mu}-h_{jk}^{\mu}h_{il}^{\mu})\,.
\end{equation}
Then we obtain
$R(\varphi,\varphi)
=\sum_{i,j,k,l}\sum_{\mu}
(h_{ik}^{\mu}h_{jl}^{\mu}-h_{jk}^{\mu}h_{il}^{\mu})\INN{\varphi_{ij}}{\varphi_{kl}}
=2h_{2}(\varphi,\varphi)$.
On the other hand,
\eqref{eqn:curv_Gauss}
obeys
$R_{ik}=\sum_{\mu}\left(
H^{\mu}h_{ik}
-\sum_{m}h_{im}^{\mu}h_{mk}^{\mu}
\right)$,
from which
we can derive
$\mathrm{Ric}(\varphi,\varphi)
=H(\varphi,\varphi)-h_{2}'(\varphi,\varphi)$.
Thus,
we have proved this lemma.
\end{proof}

Substituting \eqref{eqn:h2h2'RHRic}
into Lemma \ref{lem:sumdvp_sumdvpp}, (1)
we have
\begin{equation}\label{eqn:sumdvp_sumdvpp1'}
\sum_{A}\|d^{\nabla}(\iota_{V_{A}}\varphi)\|^{2}
=\|\nabla\varphi\|^{2}
+\dfrac{1}{2}R(\varphi,\varphi)
+H(\varphi,\varphi)-\mathrm{Ric}(\varphi,\varphi)\,.
\end{equation}
We also get
\begin{equation}\label{lem:h2h2'RHRic'}
H(\varphi,\varphi)
-2\mathrm{Ric}(\varphi,\varphi)+R(\varphi,\varphi)
=-H(\varphi,\varphi)+2(h_{2}(\varphi,\varphi)+h_{2}'(\varphi,\varphi))\,.
\end{equation}

The Bochner-Weitzenb\"ock formula
gives a way to calculate
the differential of an $F$-harmonic form.
Indeed,
we make use of this formula
to prove the following theorem,
which is a generalization of
\cite[Lemma 8]{JZ12}.

\begin{pro}[]\label{pro:FHarm_BW}
For any $F$-harmonic form
$\varphi\in\Omega^{2}(\mathfrak{g}_{P})$,
we have:
\begin{multline}\label{eqn:FHarm_BW0}
\int_{M}F''(\dfrac{1}{2}\|\varphi\|^{2})\|\varphi\|^{2}\|\nabla\|\varphi\|\|^{2}dv
+\int_{M}F'(\dfrac{1}{2}\|\varphi\|^{2})\left\{
\INN{\mathfrak{R}^{\nabla}(\varphi)}{\varphi}+
\|\nabla\varphi\|^{2}
\right\}dv\\
=-\int_{M}F'(\dfrac{1}{2}\|\varphi\|^{2})\INN{\varphi\circ(\mathrm{Ric}\wedge I+2R)}{\varphi}dv\,.
\end{multline}
\end{pro}

\begin{proof}
Let $x\in M$
and $(e_{1},\dotsc,e_{n})$
be an orthonormal basis
of $T_{x}M$.
We extend $(e_{1},\dotsc,e_{n})$
to a local orthonormal frame
field so that
$(De_{1})(x)=0,\dotsc,(De_{n})(x)=0$.
Then, at the point $x$,
we have
\begin{align}
 \Delta F(\frac{1}{2}\|\varphi\|^{2})&=-\sum\nabla_{e_{i}}\nabla_{e_{i}}F(\frac{1}{2}\|\varphi\|^{2})\\
 &=-F''(\dfrac{1}{2}\|\varphi\|^{2})\|\varphi\|^{2}\|\nabla\|\varphi\|\|^{2}
+F'(\dfrac{1}{2}\|\varphi\|^{2})\cdot \dfrac{1}{2}\Delta\|\varphi\|^{2}\,.\label{eqn:FHarm_BW1}
\end{align}
From
Proposition \ref{pro:BWformula}
and $d^{\nabla}\varphi=0$,
we can derive
\begin{equation}
\dfrac{1}{2}\Delta\|\varphi\|^{2}
=\INN{d^{\nabla}\delta^{\nabla}\varphi}{\varphi}
-\INN{\varphi\circ(\mathrm{Ric}\wedge I+2R)}{\varphi}-\INN{\mathfrak{R}^{\nabla}(\varphi)}{\varphi}-\|\nabla\varphi\|^{2}\,.
\end{equation}
Substituting this into
the right hand side of the second term of
\eqref{eqn:FHarm_BW1},
we obtain
\begin{multline}\label{eqn:FHarm_BW3}
\Delta
F(\dfrac{1}{2}\|\varphi\|^{2})\\
=-F''(\dfrac{1}{2}\|\varphi\|^{2})\|\varphi\|^{2}\|\nabla\|\varphi\|\|^{2}
-F'(\dfrac{1}{2}\|\nabla\varphi\|^{2})
\left\{
\INN{\mathfrak{R}^{\nabla}(\varphi)}{\varphi}+\|\nabla\varphi\|^{2}
\right\}\\
-F'(\dfrac{1}{2}\|\varphi\|^{2})\INN{\varphi\circ(\mathrm{Ric}\wedge I+2R)}{\varphi}
+F'(\dfrac{1}{2}\|\nabla\varphi\|^{2})
\INN{d^{\nabla}\delta^{\nabla}\varphi}{\varphi}\,.
\end{multline}

By integrating both sides
over $M$,
the left hand side
vanishes because of Green's theorem
(\cite[Appendix 6]{KN}),
and the right hand side
is equal to
\begin{multline}
-\int_{M}F''(\dfrac{1}{2}\|\varphi\|^{2})\|\varphi\|^{2}\|\nabla\|\varphi\|\|^{2}dv
-\int_{M}F'(\dfrac{1}{2}\|\nabla\varphi\|^{2})
\left\{
\INN{\mathfrak{R}^{\nabla}(\varphi)}{\varphi}+\|\nabla\varphi\|^{2}\right\}dv\\
-\int_{M}F'(\dfrac{1}{2}\|\varphi\|^{2})\INN{\varphi\circ(\mathrm{Ric}\wedge I+2R)}{\varphi}
dv\,.
\end{multline}
Here,
we have used the second equality
in \eqref{dfn:FHarm_dfn}.
From the above arguments
we have the assertion.
\end{proof}

We are ready to prove Theorem \ref{thm:sumI_cal}.

\begin{proof}[Proof of Theorem $\ref{thm:sumI_cal}$.]
Let $\varphi$ be an $F$-harmonic $2$-form.
By using
Lemmas \ref{lem_sumR_R},
\ref{lem:sumdvp_sumdvpp}, (2)
and
\eqref{eqn:sumdvp_sumdvpp1'},
the summation \eqref{eqn:sumIVA_1+}
is rewritten as follows:
\begin{multline}\label{eqn:sumIVA_1+2}
\sum_{A}I_{\varphi}(\iota_{V_{A}}\varphi)
=\int_{M}F''(\dfrac{1}{2}\|\varphi\|^{2})\INN{\bs{h}_{1}(\varphi,\varphi)}{\bs{h}_{1}(\varphi,\varphi)}dv\\
+\int_{M}F''(\dfrac{1}{2}\|\varphi\|^{2})
\|\varphi\|^{2}\|\nabla\|\varphi\|\|^{2}dv
+\int_{M}F'(\dfrac{1}{2}\|\varphi\|^{2})
\left\{
\INN{\mathfrak{R}^{\nabla}(\varphi)}{\varphi}
+\|\nabla\varphi\|^{2}
\right\}dv\\
+\int_{M}F'(\dfrac{1}{2}\|\varphi\|^{2})
\left\{
H(\varphi,\varphi)-\mathrm{Ric}(\varphi,\varphi)+\dfrac{1}{2}R(\varphi,\varphi)
\right\}dv\,.
\end{multline}
By rewriting the right hand side
of \eqref{eqn:FHarm_BW0}
by means of Lemma 
\ref{lem:RicwedgeI+2R_cal},
we obtain
\begin{multline}
\int_{M}F''(\dfrac{1}{2}\|\varphi\|^{2})
\|\varphi\|^{2}\|\nabla\|\varphi\|\|^{2}dv
+\int_{M}F'(\dfrac{1}{2}\|\varphi\|^{2})
\left\{
\INN{\mathfrak{R}^{\nabla}(\varphi)}{\varphi}
+\|\nabla\varphi\|^{2}
\right\}dv\\
=\int_{M}F'(\dfrac{1}{2}\|\varphi\|^{2})\left\{
-\mathrm{Ric}(\varphi,\varphi)
+\dfrac{1}{2}R(\varphi,\varphi)
\right\}dv\,.
\end{multline}
Substituting this
into
\eqref{eqn:sumIVA_1+2},
we can derive
\eqref{eqn:sumI_cal}.
Thus, we have complete the proof.
\end{proof}

\subsection{Analysis of the indices for \texorpdfstring{$F$}{F}-harmonic forms (2)}\label{sec:FHarm_indx_analysis2}

We will perform further calculations
of the summation
$\sum_{A}I_{\varphi}(\iota_{V_{A}}\varphi)$ in terms of 
Theorem \ref{thm:sumI_cal}.
In our calculations, the key is to evaluate the relation between $F'(\|\varphi\|^{2}/2)$
and $F''(\|\varphi\|^{2}/2)$ in \eqref{eqn:sumI_cal}.
So,
we define the degree
of the differential $F'$
as follows.

\begin{dfn}\label{dfn:F'_degree_dfn}
Let $F$ 
be a strictly increasing
$C^{2}$-function defined on
$[0,c)$,
$0<c\leq \infty$.
The \emph{degree}
of $F'$ is defined by
\begin{equation}\label{eqn:degree_dfn}
d_{F'}=\sup_{0<t<c}\dfrac{tF''(t)}{F'(t)}\,.
\end{equation}
\end{dfn}

By the definition
the degree $d_{F'}$
is valued in $\mathbb{R}\cup\{\infty\}$.

\begin{ex}\label{exFYM_F'deg}
We determine the degree
$d_{F'}$
for the functions
$F$ as in Example \ref{ex:FYM_ex}.

(1) For $F=F_{p}$ ($p\geq 2$),
we have 
\begin{equation}
d_{F_{p}'}
=\sup_{t>0}
\dfrac{tF_{p}''(t)}{F_{p}'(t)}
=\sup_{t>0}\dfrac{(p-2)t(2t)^{\frac{p-4}{2}}}{(2t)^{\frac{p-2}{2}}}
=\dfrac{p-2}{2}\,.
\end{equation}

(2) For $F=F_{\epsilon}$
($\epsilon=\pm 1$),
which is defined
on $[0,\infty)$
if $\epsilon=1$;
on $[0,1/2)$ if $\epsilon=-1$,
from
\begin{equation}
\dfrac{tF''_{\epsilon}(t)}{F'_{\epsilon}(t)}=-\dfrac{1}{2}+\dfrac{1}{2(1+2\epsilon t)}\,,
\end{equation}
we get
$d_{F_{\epsilon=1}'}=0$
and
$d_{F_{\epsilon=-1}'}=\infty$.

(3) For $F=F_{e}$, we have $F_{e}'(t)=F_{e}''(t)=e^{t}$.
Hence we have $d_{F_{e}'}=\infty$.
\end{ex}

In what follows,
we assume that the degree $d_{F'}$
is finite.
Let $\varphi\in\Omega^{2}(\mathfrak{g}_{P})$ be a non-zero, $F$-harmonic form.
The norm $\|\varphi\|$ of $\varphi$
gives a smooth function on $M$.
We define a closed subset
$M_{0}$ in $M$ as follows:
\begin{equation}
M_{0}=\left\{
x\in M \mid
\|\varphi\|(x)=0
\right\}\,.
\end{equation}
Then it is verified that $M_{0}$
has measure zero in $M$
by means of
the connectedness of $M$
and $\varphi\not\equiv 0$.
Since $d_{F'}$
is finite,
we have
\begin{equation}
F''(\dfrac{1}{2}\|\varphi\|^{2})\leq \dfrac{2}{\|\varphi\|^{2}}F'(\dfrac{1}{2}\|\varphi\|^{2})\cdot d_{F'}
\qquad
\text{on $M-M_{0}$}\,.
\end{equation}
Hence
\eqref{eqn:sumI_cal}
yields
\begin{align}
\sum_{A}I_{\varphi}(\iota_{V_{A}}\varphi)
&= \int_{M-M_{0}}F''(\dfrac{1}{2}\|\varphi\|^{2})\INN{\bs{h}_{1}(\varphi,\varphi)}{\bs{h}_{1}(\varphi,\varphi)}dv\notag\\
&\phantom{ho}+\int_{M-M_{0}}F'(\dfrac{1}{2}\|\varphi\|^{2})\{H(\varphi,\varphi)-2\mathrm{Ric}(\varphi,\varphi)+R(\varphi,\varphi)\}dv\\
&\leq\int_{M-M_{0}}\dfrac{2}{\|\varphi\|^{2}}F'(\dfrac{1}{2}\|\varphi\|^{2})\cdot d_{F'}\INN{\bs{h}_{1}(\varphi,\varphi)}{\bs{h}_{1}(\varphi,\varphi)}dv\notag\\
&\phantom{ho}+\int_{M-M_{0}}F'(\dfrac{1}{2}\|\varphi\|^{2})\{H(\varphi,\varphi)-2\mathrm{Ric}(\varphi,\varphi)+R(\varphi,\varphi)\}dv\,.\label{eqn:sumIF_ineq1}
\end{align}
So, if we put
\begin{align}
\!\!\!\!\!\!\!\!\!\!\!\!B(\varphi,\varphi)
&=d_{F'}\INN{\bs{h}_{1}(\varphi,\varphi)}{\bs{h}_{1}(\varphi,\varphi)}
+\dfrac{\|\varphi\|^{2}}{2}\{H(\varphi,\varphi)-2\mathrm{Ric}(\varphi,\varphi)+R(\varphi,\varphi)\}\\
&=d_{F'}\INN{\bs{h}_{1}(\varphi,\varphi)}{\bs{h}_{1}(\varphi,\varphi)}
+\dfrac{\|\varphi\|^{2}}{2}\{-H(\varphi,\varphi)+2(h_{2}(\varphi,\varphi)+h_{2}'(\varphi,\varphi))\}\,,\label{eqn:Bvpvp_dfn}
\end{align}
then, by \eqref{eqn:sumIF_ineq1},
we obtain
\begin{equation}\label{eqn:sumIF_ineq2}
\sum_{A}I(\iota_{V_{A}}\varphi)\leq \int_{M-M_{0}}\dfrac{2}{\|\varphi\|^{2}}F'(\dfrac{1}{2}\|\varphi\|^{2})B(\varphi,\varphi)dv\,.
\end{equation}

From the above argument,
we conclude:

\begin{thm}\label{thm:sumIF_ineq2}
Let $M$ be a connected, closed Riemannian manifold isometrically immersed in $\mathbb{R}^{N}$.
Assume that
the degree $d_{F'}$
is finite.
Then, for any non-zero,
$F$-harmonic $2$-form $\varphi$,
we have the inequality
\eqref{eqn:sumIF_ineq2}.
Furthermore,
if 
$B(\varphi,\varphi)<0$ holds,
then we have $\sum_{A}I_{\varphi}(\iota_{V_{A}}\varphi)<0$.
\end{thm}

Here, we remark that
$B(\varphi,\varphi)$
is independent of the choice
of orthonormal bases
$(e_{1},\dotsc,e_{n})$ of 
$T_{x}M$
and $(e_{n+1},\dotsc,e_{N})$
of $T^{\perp}_{x}M$.
In particular,
the inequality
$B(\varphi,\varphi)<0$
is invariant
under the orthonormal basis changes.

\subsection{Instability of \texorpdfstring{$F$}{F}-Yang-Mills connections over convex hypersurfaces in Euclidean spaces}\label{sec:FYM_inst_CCH}

Let $\varphi$ be an $F$-harmonic $2$-form.
Let $M$ be a connected, compact, convex hypersurface in an ($n+1$)-Euclidean space
$\mathbb{R}^{n+1}$
and $\lambda_{1},\dotsc,\lambda_{n}$ be its principal curvatures.
Without loss of generalities,
we may assume that
$\lambda_{i}$ is positive for each $i$.
It follows from $h^{n+1}_{ij}=\lambda_{i}\delta_{ij}$
that $H(\varphi,\varphi)$
is expressed as follows:
\begin{equation}
H(\varphi,\varphi)=\sum_{i,j}\left(\sum_{m}\lambda_{m}\right)\lambda_{i}\|\varphi_{ij}\|^{2}\,.\label{eqn:Hphph_CCC}
\end{equation}
Furthermore,
we have
%\begin{align}
$\INN{\bs{h}_{1}(\varphi,\varphi)}{\bs{h}_{1}(\varphi,\varphi)}=\sum_{i,j,i',j'}\lambda_{i}\lambda_{i'}\|\varphi_{ij}\|^{2}\|\varphi_{i'j'}\|^{2}$,
$h_{2}(\varphi,\varphi)=\sum_{i,j}\lambda_{i}\lambda_{j}\|\varphi_{ij}\|^{2}$
and
$h_{2}'(\varphi,\varphi)=\sum_{i,j}\lambda_{i}^{2}\|\varphi_{ij}\|^{2}$.
Substituting
these
into \eqref{eqn:Bvpvp_dfn}, we get:
\begin{align}
B(\varphi,\varphi)&=\sum_{i,j,i',j'}
B_{iji'}\|\varphi_{ij}\|^{2}\|\varphi_{i'j'}\|^{2}\,,\\
B_{iji'}&=d_{F'}\lambda_{i}\lambda_{i'}+\dfrac{1}{4}\left\{
-\left(\sum_{m}\lambda_{m}\right)\lambda_{i}+2\lambda_{i}\lambda_{j}+2\lambda_{i}^{2}
\right\}\,.
\end{align}
If $B_{ijk}$ is negative for each $i,j,k$,
then we obtain $B(\varphi,\varphi)<0$.
Then, 
Theorem \ref{thm:sumIF_ineq2}
yields
\begin{equation}\label{eqn:sumAIF_negative}
\sum_{A}I_{\varphi}(\iota_{V_{A}}\varphi)<0\,.
\end{equation}
On the other hand,
the inequality $B_{ijk}<0$ is rewritten as
\begin{equation}\label{eqn:FYM_inst_cdt0}
\lambda_{i}\sum_{m\neq i,j}\lambda_{m}>\lambda_{i}\left(\lambda_{i}+\lambda_{j}+4d_{F'}\lambda_{k}\right)\,,
\end{equation}
that is,
\begin{equation}\label{eqn:FYM_inst_cdt}
\sum_{m\neq i,j}\lambda_{m}>\lambda_{i}+\lambda_{j}+4d_{F'}\lambda_{k}\,.
\end{equation}
From this argument,
\eqref{eqn:FYM_inst_cdt}
gives a sufficient condition that
any non-flat, $F$-Yang-Mills connection
over $M$ is instable.
In order to prove this,
we assume for contradiction
that
there exists a non-flat, weakly stable $F$-Yang-Mills connection $\nabla$ over $M$.
Applying $\varphi=R^{\nabla}\in\Omega^{2}(\mathfrak{g}_{P})$ to \eqref{eqn:sumAIF_negative},
we have
\begin{equation}
\sum_{A}I_{R^{\nabla}}(\iota_{V_{A}}R^{\nabla})<0\,.
\end{equation}
On the other hand,
it follows from the weak instability of $\nabla$
that $I_{R^{\nabla}}(\iota_{V_{A}}R^{\nabla})\geq 0$ holds
for each $A$.
This obeys
\begin{equation}
\sum_{A}I_{R^{\nabla}}(\iota_{V_{A}}R^{\nabla})\geq 0\,,
\end{equation}
which is a contradiction.
Therefore we have derived the following theorem.

\begin{thm}\label{thm:FYM_inst_cch}
Let $\lambda_{1},\dotsc,\lambda_{n}$
be the principal curvatures of
a connected, compact, convex hypersurface $M$
in $\mathbb{R}^{n+1}$.
Assume that the degree $d_{F'}$
is finite.
Then,
any non-flat, $F$-Yang-Mills connection over $M$ is instable
if the following condition holds:
\begin{equation}\label{eqn:FYM_inst_cdt_cptconvex}
\sum_{m\neq i,j}\lambda_{m}>\lambda_{i}+\lambda_{j}+4d_{F'}\lambda_{k}
\quad
(1\leq i,j,k\leq n)\,.
\end{equation}
\end{thm}

Let us consider the case when $M$
is the standard $n$-sphere $S^{n}\subset \mathbb{R}^{n+1}$.
If we denote by $r$
the radius of $S^{n}$,
then the principal curvatures
$\lambda_{i}$ are equal to $1/r$.
Hence
we have the following result as a corollary of Theorem \ref{thm:FYM_inst_cch}.
\begin{cor}\label{cor:FYM_Sninstable}
If the inequality
\begin{equation}\label{eqn:FYM_inst_cdt_sphere}
n>4d_{F'}+4
\end{equation}
holds, then 
any non-flat, $F$-Yang-Mills connection
over $S^{n}$ is instable.
\end{cor}

We give an application
of
Corollary
\ref{cor:FYM_Sninstable}
for $F$-Yang-Mills connections
as in Example \ref{ex:FYM_ex}.
As shown in Example \ref{exFYM_F'deg},
for $F=F_{p}$
($p\geq 2$),
$F_{\epsilon=1}$,
the degree $d_{F'}$ is finite.

(1) In the case of $F=F_{p}$,
\eqref{eqn:FYM_inst_cdt_sphere}
reduces to $n>2p$.
Hence,
if $n>2p$, then
any non-flat, $p$-Yang-Mills connection
over $S^{n}$
is instable.
This result coincides with
the results of Simons (\cite{Simons}) for $p=2$
and Chen-Zhou (\cite[Corollary 4.2]{CZ}) for $p\geq 2$.

(2) In the case of $F=F_{\epsilon=1}$,
we have obtained $d_{F_{\epsilon=1}'}=0$.
Thus,
if $n>4$,
then any non-flat, critical connection
of the generalized Yang-Mills-Born-Infeld
energy functional with positive sign
is instable.

By means of Theorem \ref{thm:sumI_cal},
we give an observation
for the instability
of an  $F$-Yang-Mills connection
in the case when
$F'$ has infinite degree.
Here, we recall that
Theorem \ref{thm:sumI_cal}
does not require
no assumptions about the finiteness of $d_{F'}$.
Now, 
let us consider
the instability
of critical connections
of $\mathscr{Y\!\!M}_{\epsilon=-1}$
and exponential Yang-Mills connections,
which are examples of 
$F$-Yang-Mills connections with $d_{F'}=\infty$.
We first consider the case when $F=F_{\epsilon=-1}$.
Based on the domain of
definition for $F_{\epsilon=-1}$,
we assume
that the $F_{\epsilon=-1}$-harmonic form
$\varphi=R^{\nabla}\in\Omega^{2}(\mathfrak{g}_{P})$
satisfies
$\|\varphi\|<1$.
From $h_{ij}^{n+1}=(1/r)\delta_{ij}$
we get
\begin{equation}
H(\varphi,\varphi)
=\dfrac{2n}{r^{2}}\|\varphi\|^{2},\quad
\INN{\bs{h}_{1}(\varphi,\varphi)}{\bs{h}_{1}(\varphi,\varphi)}=\dfrac{4}{r^{2}}\|\varphi\|^{4},
\quad
h_{2}(\varphi,\varphi)
=h_{2}'(\varphi,\varphi)
=\dfrac{2}{r^{2}}\|\varphi\|^{2}\,.
\end{equation}
By Theorem \ref{thm:sumI_cal}
and \eqref{lem:h2h2'RHRic'},
we have
\begin{align}
&\sum_{A}(I_{\varphi}\iota_{V_{A}}\varphi)\\
&\phantom{hoge}=\int_{S^{n}}\dfrac{4}{r^{2}}F_{\epsilon=-1}''(\dfrac{1}{2}\|\varphi\|^{2})\|\varphi\|^{4}dv
+\int_{S^{n}}F_{\epsilon=-1}'(\dfrac{1}{2}\|\varphi\|^{2})
\left\{-\dfrac{2(n-4)}{r^{2}}\right\}\|\varphi\|^{2}dv\,.
\end{align}
By using 
\begin{equation}
F_{\epsilon=-1}'(\dfrac{1}{2}\|\varphi\|^{2})=\dfrac{1}{\sqrt{1-\|\varphi\|^{2}}}
\,,\quad
F_{\epsilon=-1}''(\dfrac{1}{2}\|\varphi\|^{2})=\dfrac{1}{(1-\|\varphi\|^{2})\sqrt{1-\|\varphi\|^{2}}}\,,
\end{equation}
we obtain
\begin{equation}\label{eqn:gYMBI_indxsum}
\sum_{A}I_{\varphi}(\iota_{V_{A}}\varphi)
=\dfrac{2}{r^{2}}\int_{S^{n}}
\dfrac{\|\varphi\|^{2}}{\sqrt{1-\|\varphi\|^{2}}}
\left\{
\dfrac{2}{1-\|\varphi\|^{2}}
-(n-2)
\right\}dv\,.
\end{equation}
From this argument,
if the integrand of the right hand side
of \eqref{eqn:gYMBI_indxsum}
is negative on $S^{n}$,
then $\sum_{A}I_{\varphi}(\iota_{V_{A}}\varphi)<0$ holds.
Thus, we derive the following proposition.

\begin{pro}\label{pro:gYMBI-_inst}
Let $\nabla$
be a non-flat,
critical connection 
over the standard $n$-sphere $S^{n}$
for the generalized Yang-Mills-Born-Infeld
energy functional
$\mathscr{Y\!\!M}_{\epsilon=-1}$
with negative sign.
If $n>4$ and
the curvature $2$-form $R^{\nabla}$
satisfies
\begin{equation}
\|R^{\nabla}\|<\sqrt{\dfrac{n-4}{n-2}}\,,
\end{equation}
then $\nabla$ is instable.
\end{pro}

We give
an analogous result for exponential Yang-Mills connections.
Let $\nabla$ be
an exponential Yang-Mills connection
over $S^{n}$ and $\varphi$
be in $\Omega^{2}(\mathfrak{g}_{P})$.
A similar calculation shows
\begin{align}
\sum_{A}(I_{\varphi}\iota_{V_{A}}\varphi)
&=\dfrac{2}{r^{2}}\int_{S^{n}}
\exp(\dfrac{1}{2}\|\varphi\|^{2})
\|\varphi\|^{2}
\left\{
2\|\varphi\|^{2}-(n-4)
\right\}dv\,.
\end{align}
From this we conclude:

\begin{pro}\label{pro:eYM-_inst}
Let $\nabla$
be a non-flat,
exponential Yang-Mills connection 
over the standard $n$-sphere $S^{n}$.
If $n>4$ and
the curvature $2$-form $R^{\nabla}$
satisfies
\begin{equation}
\|R^{\nabla}\|
<\sqrt{\dfrac{n-4}{2}}\,,
\end{equation}
then $\nabla$ is instable.
\end{pro}

There are strong similarities
between the theory of Yang-Mills connections and that of harmonic maps,
which are critical points of a certain
energy functional defied on the space
of smooth map between Riemannian manifolds.
Finally,
we discuss a counter part of our results in the theory
of harmonic maps as follows:
Ara \cite{Ara}
introduced the notion of $F$-harmonic maps as a generalization of harmonic maps, $p$-harmonic maps and so on.
He (\cite[Theorem 7.1]{Ara}) also derived the instability theorem
of $F$-harmonic maps
from a closed Riemannian manifold
into the $n$-dimensional standard sphere $S^{n}$,
which is an extension of the results
by Leung \cite{Leung} for harmonic maps
and by Cheung-Leung \cite{CL}
for $p$-harmonic maps.
By means of Ara's result,
the finiteness
of the degree $d_{F'}$
in the sense of Definition \ref{dfn:F'_degree_dfn}
yields the following
statement as a counter part of Corollary \ref{cor:FYM_Sninstable}:
If the inequality
\begin{equation}
n>2d_{F'}+2
\end{equation}
holds, then any non-constant
$F$-harmonic map from a connected,
closed Riemannian manifold
into $S^{n}$ is instable.
This inequality is a natural extension of Leung's one \cite[Corollary 1]{Leung}.
We can also find a counter part
of Proposition \ref{pro:eYM-_inst}
in the theory of exponentially harmonic maps due to Koh \cite[Theorem, p.~212]{Koh}.

\end{document}